\documentclass[11pt,a4paper]{article}
\usepackage{amsmath}
\usepackage{amsfonts}
\usepackage{amsthm}
\usepackage{amsbsy}
\usepackage{mathrsfs}
\usepackage{mathptmx}
\usepackage{tikz}
\usepackage[hmargin=1in,vmargin=1in]{geometry}
\usepackage{arcs}
\usepackage[authoryear]{natbib}
\usepackage{url}
\usepackage{lineno}

\numberwithin{equation}{section}

\newtheorem{thm}{Theorem}[section]
\newtheorem{prop}[thm]{Proposition}

\newtheorem{lem}[thm]{Lemma}

\theoremstyle{remark}
\newtheorem{rem}[thm]{Remark}

\theoremstyle{definition}
\newtheorem{defn}[thm]{Definition}
\newtheorem{ex}[thm]{Example}

\newcommand{\D}[4]{{}^{#1}_{#2}D_{#3}^{#4}}
\renewcommand{\L}{\mathcal L}
\newcommand{\N}{\mathbb N}
\newcommand{\R}{\mathbb R}

\newcommand{\ceil}[1]{\lceil #1 \rceil}
\renewcommand{\d}{\mathrm d}
\newcommand{\e}{\mathrm e}
\renewcommand{\i}{\mathrm i}

\DeclareMathOperator*{\erfc}{erfc}

\DeclareMathOperator*{\Var}{Var}

\begin{document}

\title{A unified way to solve IVPs and IBVPs for the time-fractional diffusion-wave equation}
\author{Marianito R. Rodrigo\thanks{School of Mathematics and Applied Statistics,
University of Wollongong, Wollongong, New South Wales, Australia.
E-mail: {\tt marianito\_rodrigo@uow.edu.au}}}
\date{September 3, 2021}
\maketitle

\begin{abstract}

\noindent
The time-fractional diffusion-wave equation is revisited, where the time derivative is of order~$2 \nu$ and $0 < \nu \le 1$. The behaviour of the equation is ``diffusion-like'' (respectively, ``wave-like'') when $0 < \nu \le \frac{1}{2}$ (respectively, $\frac{1}{2}  < \nu \le 1$). Two types of time-fractional derivatives are considered, namely the Caputo and Riemann-Liouville derivatives. Initial value problems and initial-boundary value problems are investigated and handled in a unified way using an embedding method. A two-parameter auxiliary function is introduced and its properties are investigated. The time-fractional diffusion equation is used to generate a new family of probability distributions, and that includes the normal distribution as a particular case.


\medskip
\noindent {\bf MSC~2020 Subject Classifications:} 26A33; 35R11; 35K05; 35L05; 60E05


\end{abstract}

\section{Introduction}

Let $x \in \R$ and $t \ge 0$. Consider the following partial differential equation~(PDE) for the function~$u(x,t)$:
\begin{equation}
\label{diff-wave-PDE}
D^{2 \nu} u = \kappa \frac{\partial^2 u}{\partial x^2}, 
\end{equation}
where $\kappa > 0$ and $0 < \nu \le 1$. The ``time-fractional derivative operator''~$D^{2 \nu}$ is to be defined such that if $\nu = \frac{1}{2}$ (respectively, $\nu = 1$), then \eqref{diff-wave-PDE} is the classical diffusion equation (respectively, wave equation). If $0 < \nu \le \frac{1}{2}$ (respectively, $\frac{1}{2} < \nu \le 1$), then we say that the behaviour of \eqref{diff-wave-PDE} is ``diffusion-like'' (respectively, ``wave-like'') and call \eqref{diff-wave-PDE} the time-fractional diffusion equation (respectively, time-fractional wave equation). More generally, when $0 < \nu \le 1$, we refer to \eqref{diff-wave-PDE} as the time-fractional diffusion-wave equation~\citep{Ma1996}. A space-fractional diffusion-wave equation has also been studied in the literature~\citep{MaPaGo2007} but will not be considered in this article.

To see how $D^{2 \nu}$ can be defined, let us review some pertinent definitions from the theory of the fractional calculus; see for instance \citep{Di2010,MiRo1993,Or2011,Po1999,SaKiMa2002} and the comprehensive references therein. For a suitable function~$y(t)$, the Riemann-Liouville fractional integral of order~$\nu > 0$ is defined as
$$
\D{}{0}{t}{-\nu} y(t) = \frac{1}{\Gamma(\nu)} \int_0^t (t - \tau)^{\nu - 1} y(\tau) \, \d \tau,
$$
where $\Gamma$ is the Euler gamma function. Let $\ceil{\nu}$ denote the least integer greater than or equal to $\nu$. It follows that $\ceil{\nu} \ge \nu > 0$. Then the Caputo fractional derivative of order~$\nu > 0$ is given by
$$
\D{C}{0}{t}{\nu} y(t) =  \D{}{0}{t}{-(\ceil{\nu} - \nu)}  D^{\ceil{\nu}}  y(t),
$$
while the Riemann-Liouville fractional derivative of order~$\nu > 0$ is defined as
$$
\D{}{0}{t}{\nu} y(t) = D^{\ceil{\nu}} \D{}{0}{t}{-(\ceil{\nu} - \nu)} y(t).
$$
Observe that $\D{}{0}{t}{-(\ceil{\nu} - \nu)}$ is a Riemann-Liouville fractional integral operator of order~$(\ceil{\nu} - \nu)$ and $D^{\ceil{\nu}}$ is the ordinary derivative operator of order~$\ceil{\nu}$. If $\nu = n \in \N$, then the Riemann-Liouville fractional integral becomes $n$-fold integration, while the Caputo and Riemann-Liouville fractional derivatives reduce to $n$-fold differentiation. When $0 < \nu \le 1$, it can be shown~\citep{Di2010} that the Caputo and Riemann-Liouville fractional derivatives are related by
\begin{equation}
\label{caputo-riemann-liouville}
\D{C}{0}{t}{\nu} y(t) = \D{}{0}{t}{\nu} y(t) - \frac{y(0+)}{\Gamma(1 - \nu)} t^{-\nu}.
\end{equation}
In this article we will study initial value problems~(IVPs) and initial-boundary value problems~(IBVPs) associated with \eqref{diff-wave-PDE} both when $D^{2 \nu} = \D{C}{0}{t}{2 \nu}$ and $D^{2 \nu} = \D{}{0}{t}{2 \nu}$. Note that $D^n u(x,t)$ when $n \in \N$ is just the $n$th partial derivative of $u$ with respect to $t$.

The Caputo time-fractional diffusion equation~\eqref{diff-wave-PDE}, where $D^{2 \nu} = \D{C}{0}{t}{2 \nu}$ and $0 < \nu \le \frac{1}{2}$, was considered by \citet{Ni1986} to describe diffusion in media with fractal geometry, i.e. in special types of porous media. \citet{Ma1993} observed that the Caputo time-fractional wave equation~\eqref{diff-wave-PDE}, where $D^{2 \nu} = \D{C}{0}{t}{2 \nu}$ and $\frac{1}{2} < \nu \le 1$, governs the propagation of mechanical diffusive waves in viscoelastic media which exhibit a power-law creep. In fact, time-fractional derivatives are expected to arise when hereditary mechanisms of power-law type are present in diffusion or wave phenomena~\citep{Ma1996}. More recently, \citet{WeChZh2015} developed a Caputo time-fractional diffusion model to decribe how chloride ions penetrate reinforced concrete structures exposed to chloride environments. 

Let $f(x)$ and $h(t)$ be given suitable functions. The Cauchy problem for the Caputo time-fractional diffusion-wave equation was studied in \citep{Ma1993,Ma1996,Ma2012,MaPaGo2007}. It is an IVP of the form
\begin{equation}
\label{mainardi-cauchy-diff}
\begin{split}
& \D{C}{0}{t}{2 \nu} u = \kappa \frac{\partial^2 u}{\partial x^2}, \quad x \in \R, \quad t > 0, \\
& u(x,0+) = f(x), \quad x \in \R
\end{split}
\end{equation}
when $0 < \nu \le \frac{1}{2}$ and
\begin{equation}
\label{mainardi-cauchy-wave}
\begin{split}
& \D{C}{0}{t}{2 \nu} u = \kappa \frac{\partial^2 u}{\partial x^2}, \quad x \in \R, \quad t > 0, \\
& u(x,0+) = f(x), \quad D^1 u(x,0+) = 0, \quad x \in \R
\end{split}
\end{equation}
when $\frac{1}{2} < x \le 1$. Moreover, the signalling problem was also considered in \citep{Ma1993,Ma1996,Ma2012,MaPaGo2007}, i.e. an IVBP of the form
\begin{equation}
\label{mainardi-signal-diff}
\begin{split}
& \D{C}{0}{t}{2 \nu} u = \kappa \frac{\partial^2 u}{\partial x^2}, \quad x > 0, \quad t > 0, \\
& u(x,0+) = 0, \quad x > 0, \\
& u(0+,t) = h(t), \quad t > 0
\end{split}
\end{equation}
when $0 < \nu \le \frac{1}{2}$ and
\begin{equation}
\label{mainardi-signal-wave}
\begin{split}
& \D{C}{0}{t}{2 \nu} u = \kappa \frac{\partial^2 u}{\partial x^2}, \quad x > 0, \quad t > 0, \\
& u(x,0+) = 0, \quad D^1 u(x,0+) = 0, \quad x > 0, \\
& u(0+,t) = h(t), \quad t > 0
\end{split}
\end{equation}
when $\frac{1}{2} < \nu \le 1$. \citet{Ma1996} showed that the fundamental solutions of the Cauchy and signalling problems can be expressed in terms of an auxiliary function~$M(z;\nu)$ of a similarity variable~$z = \frac{\vert x \vert}{t^\nu}$. 

\citet{MaPaGo2007} gave two generalisations of the classical diffusion equation by replacing either the time derivative by a Caputo time-fractional derivative or the space derivative by an appropriate pseudo-differential operator (thus obtaining a symmetric space-fractional diffusion equation). They demonstrated how the fundamental solutions of these generalised equations for the Cauchy and signalling problems provide probability density functions related to so-called stable distributions. 

\begin{rem}
\label{mainardi-ass}
It should be noted that in \citep{Ma1993,Ma1996,Ma2012,MaPaGo2007} the solutions of the Cauchy and signalling problems were assumed to decay to zero at infinity (i.e.~$u(\pm \infty,t) = 0$ in the Cauchy problem and $u(\infty,t) = 0$ in the signalling problem). Furthermore, the special initial condition~$D^1 u(x,0+) = 0$ when $\frac{1}{2} < \nu \le 1$ was chosen to ensure the continuous dependence of the solution with respect to the parameter~$\nu$ as $\nu \rightarrow \frac{1}{2}^\pm$.
\end{rem}

\citet{GoReRoSa2015} studied two IBVPs associated with a Caputo time-fractional diffusion equation on the half-line. They considered either a Dirichlet or a Neumann boundary condition~(BC) as $x \rightarrow 0^+$. These IBVPs were solved analytically by taking into account the asymptotic behaviour and the existence of bounds of the Mainardi and Wright special functions.

In this article we investigate IVPs and IBVPs for \eqref{diff-wave-PDE} in a unified way using the embedding approach introduced by \citet{RoTh2021}. Most results in the literature, such as those mentioned above, have focused on the Caputo fractional derivative primarily because it is associated with initial data that can be physically measured, e.g. $u(x,0+)$ and $D^1 u(x,0+)$ could represent the initial position and velocity, respectively. For the Riemann-Liouville derivative it is not clear what the initial data should look like and it is of theoretical interest to study \eqref{diff-wave-PDE} also for this type of derivative and compare the results obtained with those from the Caputo derivative. We will let the Laplace transform indicate what the appropriate initial conditions should be. This is akin to the idea used by \citet{Ro2016,Ro2020} to determine the correct initial conditions for the ``fractional'' analogues of the matrix exponential. 

We start by considering an IVP for an inhomogeneous time-fractional diffusion-wave equation, namely
\begin{equation}
\label{combined-IVP-diff}
\begin{split}
& D^{2 \nu} u = \kappa \frac{\partial^2 u}{\partial x^2} + F(x,t), \quad x \in \R, \quad t > 0, \\
& \Phi u(x,0+) = f(x), \quad x \in \R
\end{split}
\end{equation}
when $0 < \nu \le \frac{1}{2}$ and
\begin{equation}
\label{combined-IVP-wave}
\begin{split}
& D^{2 \nu} u = \kappa \frac{\partial^2 u}{\partial x^2} + F(x,t), \quad x \in \R, \quad t > 0, \\
& \Psi_1 u(x,0+) = f(x), \quad \Psi_2 u(x,0+) = g(x), \quad x \in \R
\end{split}
\end{equation}
when $\frac{1}{2} < \nu \le 1$. The operators~$\Phi$, $\Psi_1$ and $\Psi_2$ are defined by
\begin{equation}
\label{Phi-def}
\Phi u = 
\begin{cases}
u & \text{if $D^{2 \nu} = \D{C}{0}{t}{2 \nu}$}, \\
\D{}{0}{t}{-(1 - 2 \nu)} u & \text{if $D^{2 \nu} = \D{}{0}{t}{2 \nu}$},
\end{cases}
\end{equation}
\begin{equation}
\label{Psi-def}
\Psi_1 u = 
\begin{cases}
u & \text{if $D^{2 \nu} = \D{C}{0}{t}{2 \nu}$}, \\
\D{}{0}{t}{-(2 - 2 \nu)} u & \text{if $D^{2 \nu} = \D{}{0}{t}{2 \nu}$},
\end{cases} \quad
\Psi_2 u = 
\begin{cases}
D^1 u & \text{if $D^{2 \nu} = \D{C}{0}{t}{2 \nu}$}, \\
D^1 \D{}{0}{t}{-(2 - 2 \nu)} u & \text{if $D^{2 \nu} = \D{}{0}{t}{2 \nu}$}.
\end{cases}
\end{equation}
Here we assume that $F(x,t)$, $f(x)$ and $g(x)$ are given suitable functions. The justification for the initial conditions will be given in the next section after we recall the Laplace transforms of the fractional operators. If $F(x,t) = 0$, $D^{2 \nu} = \D{C}{0}{t}{2 \nu}$ and $g(x) = 0$, then the Cauchy problem considered previously by Mainardi and other authors is recovered. 

After obtaining the solution of the above IVP, we then turn our attention to an IVBP for a (homogeneous) time-fractional diffusion-wave equation. That is, we consider \eqref{combined-IVP-diff} and \eqref{combined-IVP-wave} with $F(x,t) = 0$ but replace $x \in \R$ by $x \ge 0$, and impose a BC as $x \rightarrow 0^+$. In this paper we focus on a Dirichlet-type BC of the form
$$
u(0+,t) = h(t), \quad t > 0
$$ 
for a suitable function~$h(t)$, although other types of linear BCs can also be considered as in \citep{RoTh2021}. The embedding method of \cite{RoTh2021} involves embedding the homogeneous PDE and initial conditions of the IBVP for $x \ge 0$ into an IVP with an inhomogeneous PDE for $x \in \R$ and using the previous result for an IVP. After solving this ``enlarged'' IVP, we then restrict the solution~$u(x,t)$ to $x \ge 0$. The embedding step introduces an arbitrary function, say $\varphi(t)$. So far, $u(x,t)$ satisfies the PDE and initial conditions of the IBVP. The last step is to determine $\varphi(t)$ by imposing the BC~$u(0+,t) = h(t)$.

\begin{rem}
To preserve generality and ensure the adaptability of our results to other contexts, we relax the assumptions stated in Remark~\ref{mainardi-ass}. Instead of imposing decay properties for the solution, we assume that $u(\pm \infty,t) < \infty$ for the IVP and $u(\infty,t) < \infty$ for the IVBP. Moreover, for $\frac{1}{2} < \nu \le 1$, we replace the initial condition~$D^1 u(x,0+) = 0$ by $D^1 u(x,0+) = g(x)$. Of course, in the special case when $g(x) = 0$, the continuous dependence of the solution with respect to the parameter~$\nu$ as $\nu \rightarrow \frac{1}{2}^\pm$ is reestablished. Lastly, the aim of this article is to derive the formal solutions of the proposed IVP and IBVP; rigorous justification of these formulas in terms of the appropriate function spaces etc. is outside the scope of this work and will be treated in a future article. 
\end{rem}

The structure of this article is as follows. In Section~2 we introduce a two-parameter auxiliary function that will be used throughout the paper and investigate some of its properties. In Section~3 we consider an IVP for the Caputo time-fractional diffusion-wave equation on the real line. An analogous IVP for a Riemann-Liouville time-fractional diffusion-wave equation on the real line is studied in Section~4. Section~5 considers IBVPs for the Caputo and Riemann-Liouville time-fractional diffusion-wave equations on the half-line using the embedding approach introduced by the author. In Section~6 we discuss how the time-fractional diffusion equation can be used to generate a new family of probability distributions. Brief concluding remarks are given in Section~7. The derivation of an integral representation of the auxiliary function in the case~$0 < \nu \le \frac{1}{2}$ is given in the Appendix.

\section{Useful formulas and preliminary results}

Recall that the Laplace transform of $y(t)$ is defined as
$$
\hat y(s) = \L\{y(t);s\} = \int_0^\infty \e^{-s t} y(t) \, \d t,
$$
provided the improper integral converges at $s$. It can be shown~\citep{Di2010,Po1999} that for any $\nu > 0$, we have
\begin{subequations}
\begin{align}
\label{lap-prop1}
\L\{\D{}{0}{t}{-\nu} y(t);s\} & = s^{-\nu} \hat y(s), \\
\label{lap-prop2}
\L\{\D{C}{0}{t}{\nu} y(t);s\} & = s^\nu \hat y(s) - \sum_{j = 0}^{\ceil{\nu} - 1} s^{\nu - 1 - j} D^j y(0+), \\
\label{lap-prop3}
\L\{\D{}{0}{t}{\nu} y(t);s\} & = s^\nu \hat y(s) - \sum_{j = 0}^{\ceil{\nu} - 1} s^{\ceil{\nu} - 1 - j} D^j \D{}{0}{t}{-(\ceil{\nu} - \nu)} y(0+).
\end{align}
\end{subequations}
Since $D^j y(0+)$ is usually measurable physically but not $D^j \D{}{0}{t}{-(\ceil{\nu} - \nu)} y(0+)$, the Caputo fractional derivative has been utilised more than the Riemann-Liouville fractional derivative in applications. Nevertheless, from a theoretical perspective, the initial conditions in \eqref{lap-prop2} and \eqref{lap-prop3} provide the motivation for the definitions of the operators~$\Phi$, $\Psi_1$ and $\Psi_2$ in \eqref{Phi-def} and \eqref{Psi-def}.

Let us now introduce a very useful two-parameter auxiliary function and investigate some of its properties.
\begin{defn}
Let $\mu \ge 0$, $0 < \nu \le 1$ and $a > 0$. Define the function
\begin{equation}
\label{rodrigo-fun}
R_{\mu,\nu}(a,t) = \L^{-1}\{s^{-\mu} \e^{-a s^\nu};t\}, \quad t > 0.
\end{equation}
\end{defn}
For $\nu = 0.3, 0.4, 0.5, 0.6, 0.7$, profiles of $R_{0,\nu}(2.5,t)$ and $R_{2 \nu,\nu}(2.5,t)$ are shown in Figure~1 and Figure~2, respectively.
\begin{figure}[h]
\centering
\includegraphics[scale=0.5]{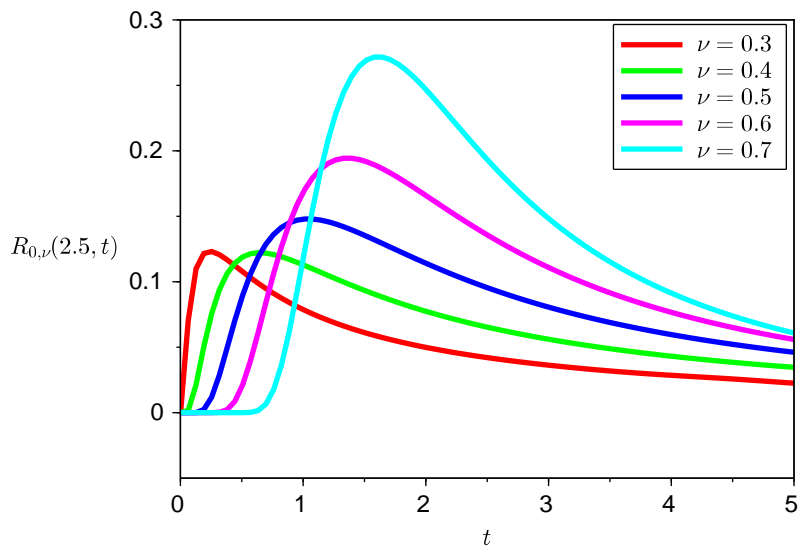}
\caption{Plot of $R_{0,\nu}(2.5,t)$ for different values of $\nu$.}
\end{figure}
\begin{figure}[ht]
\centering
\includegraphics[scale=0.5]{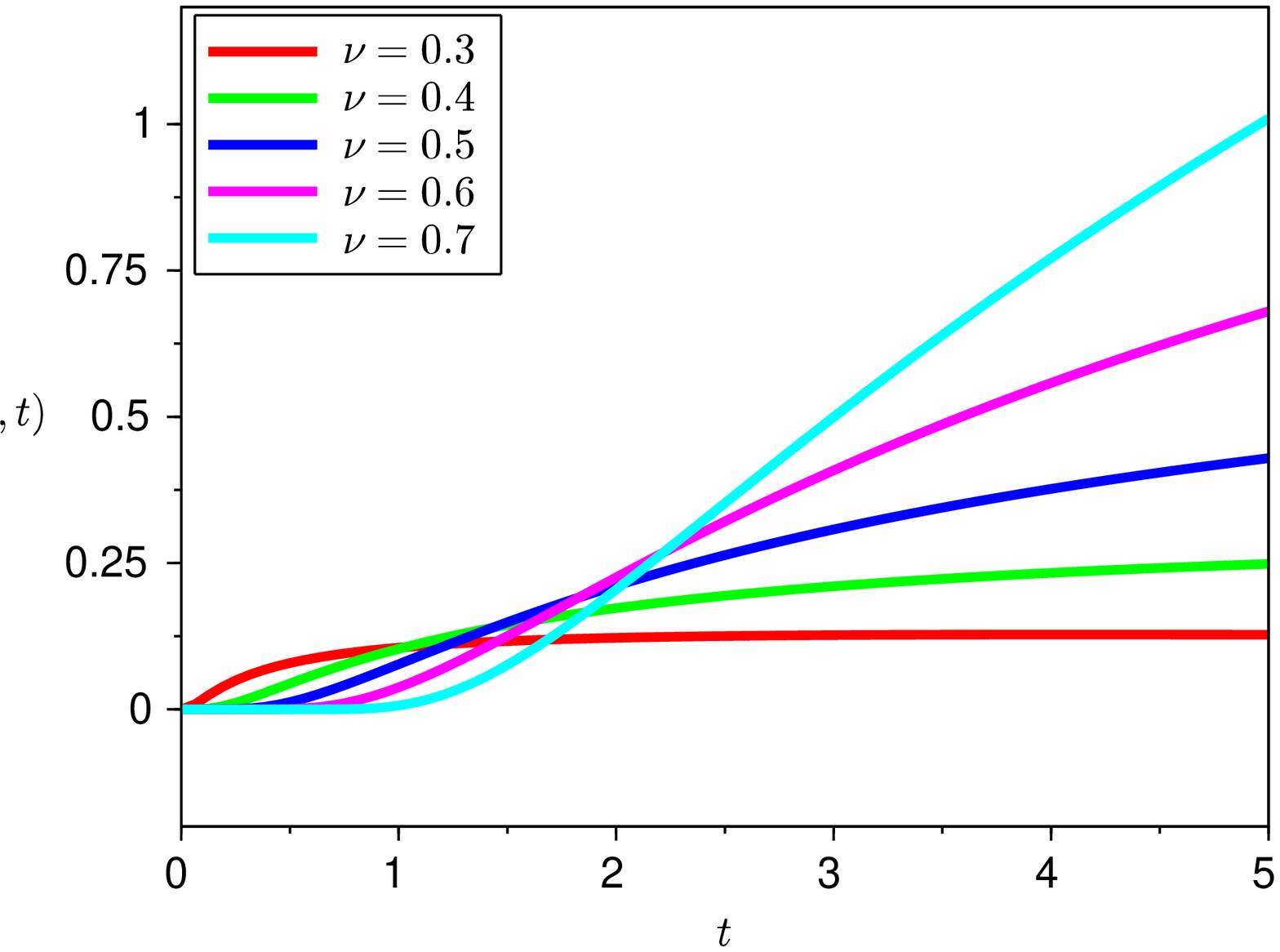}
\caption{Plot of $R_{2 \nu,\nu}(2.5,t)$ for different values of $\nu$.}
\end{figure}

\begin{rem}
It follows from a property of the Dirac delta function~$\delta$~\citep[p.~251]{Sp1965} that
$$
R_{0,1}(a,t) = \L^{-1}\{\e^{-a s};t\} = \delta(t - a).
$$
Furthermore, since $R_{0,\nu}(a,t) = \L^{-1}\{\e^{-a s^\nu};t\}$, we deduce from \eqref{lap-prop1} that
\begin{equation}
\label{rodrigo-fun-spec}
R_{\mu,\nu}(a,t) = \D{}{0}{t}{-\mu} R_{0,\nu}(a,t).
\end{equation}
\end{rem}

\begin{prop}[Properties of $R_{\mu,\nu}(a,t)$]
Suppose that $\mu \ge 0$, $0 < \nu \le 1$ and $a > 0$. Then the following properties hold:
\begin{enumerate}
\item[\rm (i)] 
\begin{equation}
\label{val-zero}
\lim_{t \rightarrow 0^+} R_{\mu,\nu}(a,t) = 0.
\end{equation}
\item[\rm (ii)] The function~$y(t) = R_{\mu,\nu}(a,t)$ satisfies the fractional integral equation
\begin{equation}
\label{rodrigo-int-eq}
a \nu \D{}{0}{t}{-(1 - \nu)} y(t) = t y(t) - \mu \int_0^t y(\tau) \, \d \tau
\end{equation}
and the fractional ordinary differential equation
\begin{equation}
\label{rodrigo-fode}
a \nu \D{}{0}{t}{\nu} y(t) = a \nu \D{C}{0}{t}{\nu} y(t) = t y'(t) + (1 - \mu) y(t).
\end{equation}
\end{enumerate}
\end{prop}
\begin{proof}
(i) Applying the initial value theorem for the Laplace transform and using \eqref{rodrigo-fun}, we obtain
$$
\lim_{t \rightarrow 0^+} R_{\mu,\nu}(a,t) = \lim_{s \rightarrow \infty} s \L\{R_{\mu,\nu}(a,t);s\} = \lim_{s \rightarrow \infty} s^{1 - \mu} \e^{-a s^\nu} = 0.
$$

(ii) It follows from (i) and \eqref{caputo-riemann-liouville} that $\D{C}{0}{t}{\nu} y(t) = \D{}{0}{t}{\nu} y(t)$. Moreover, \eqref{rodrigo-fun} implies that $\hat y(s) = s^{-\mu} \e^{-a s^\nu}$ and therefore
$$
(\hat y)'(s) = -\mu s^{-\mu - 1} \e^{-a s^\nu} + s^{-\mu} \e^{-a s^\nu} (-a \nu s^{\nu - 1}) = -\mu \frac{\hat y(s)}{s} - a \nu s^{-(1 - \nu)} \hat y(s),
$$
which is equivalent to
$$
\L\{-t y(t);s\} = -\mu \L\Big\{\int_0^t y(\tau);s\Big\} - a \nu \L\{\D{}{0}{t}{-(1 - \nu)} y(t);s\}
$$
using standard properties of the Laplace transform and \eqref{lap-prop1}. Thus we derive \eqref{rodrigo-int-eq}. Finally, taking the (ordinary) derivative of both sides of \eqref{rodrigo-int-eq} with respect to $t$ and recalling the definition of the Riemann-Liouville fractional derivative, we get \eqref{rodrigo-fode}. We remark that Figures~1 and 2 were generated using the numerical Laplace inversion of \eqref{rodrigo-fun}. An alternative is to numerically solve either the fractional integral equation~\eqref{rodrigo-int-eq} or the fractional ordinary differential equation~\eqref{rodrigo-fode}.
\end{proof}

The next result gives an integral representation of $R_{0,\nu}(a,t)$ when $0 < \nu \le \frac{1}{2}$. The proof is given in the Appendix.
\begin{prop}
If $0 < \nu \le \frac{1}{2}$ and $a > 0$, then 
\begin{equation}
\label{rodrigo-real-int}
R_{0,\nu}(a,t) = \frac{1}{\pi} \int_0^\infty \e^{-t x} \e^{-a \cos(\pi \nu) x^\nu} \sin(a \sin (\pi \nu) x^\nu) \, \d x, \quad t > 0.
\end{equation}
\end{prop}

\begin{rem}
A particular case when the integral in \eqref{rodrigo-real-int} can be evaluated occurs when $\nu = \frac{1}{2}$. We shall see later that this integral is related to the Green's function for the IVP for the classical diffusion equation. Indeed, \eqref{rodrigo-real-int} can be expressed as
$$
R_{0,\frac{1}{2}}(a,t) = \frac{1}{\pi} \int_0^\infty \e^{-t x} \sin(a \sqrt{x}) \, \d x = \frac{1}{\pi} \L\{\sin(a \sqrt{x});t\}.
$$
Term-by-term Laplace transformation of the Taylor series expansion
$$
\sin(a \sqrt{x}) = \sum_{j = 0}^\infty (-1)^j \frac{a^{2 j + 1} x^{j + \frac{1}{2}}}{(2 j + 1)!}
$$
and the formula~$\L\{x^p;t\} = \frac{\Gamma(p + 1)}{t^{p + 1}}$ for $p > -1$ yields
\begin{align*}
\frac{1}{\pi} \L\{\sin(a \sqrt{x});t\} & = \frac{1}{\pi} \sum_{j = 0}^\infty (-1)^j \frac{a^{2 j + 1}}{(2 j + 1)!} \L\{x^{j + \frac{1}{2}};t\} = \frac{1}{\pi} \sum_{j = 0}^\infty (-1)^j \frac{a^{2 j + 1} \Gamma(j + \frac{3}{2})}{(2 j + 1)! t^{j + \frac{3}{2}}} \\
& = \frac{a}{2 \sqrt{\pi} t^{\frac{3}{2}}} \sum_{j = 0}^\infty (-1)^j \frac{2 a^{2 j} \Gamma(j + \frac{3}{2})}{\sqrt{\pi} (2 j + 1)! t^j}.
\end{align*}
Recalling the property~$\Gamma(z + 1) = z \Gamma(z)$, we see that
$$
\frac{1}{\pi} \L\{\sin(a \sqrt{x});t\} = \frac{a}{2 \sqrt{\pi} t^{\frac{3}{2}}} \sum_{j = 0}^\infty (-1)^j \frac{a^{2 j} \Gamma(j + \frac{1}{2})}{\sqrt{\pi} (2j)! t^j}.
$$
Using the Legendre duplication formula 
$$
\sqrt{\pi} 2^{1 - 2 z} = \frac{\Gamma(z) \Gamma(z + \frac{1}{2})}{\Gamma(2 z)}
$$
gives
$$
\sqrt{\pi} 2^{1 - 2 j} = \frac{(j - 1)!  \Gamma(j + \frac{1}{2})}{(2 j - 1)!} = \frac{2 j!  \Gamma(j + \frac{1}{2})}{(2 j)!} \quad \text{or} \quad \frac{\Gamma(j + \frac{1}{2})}{\sqrt{\pi} (2j)!} = \frac{1}{2^{2 j} j!} = \frac{1}{4^j j!}.
$$
Thus
\begin{equation}
\label{rodrigo-half}
R_{0,\frac{1}{2}}(a,t) = \frac{1}{\pi} \L\{\sin(a \sqrt{x});t\} = \frac{a}{2 \sqrt{\pi} t^{\frac{3}{2}}} \sum_{j = 0}^\infty \frac{(-\frac{a^2}{4 t})^j}{j!} = \frac{a \e^{-\frac{a^2}{4 t}}}{2 \sqrt{\pi} t^{\frac{3}{2}}}.
\end{equation}
Moreover, \eqref{rodrigo-fun}, \eqref{rodrigo-fun-spec} and formula~84 in \citet[p.250]{Sp1965} give
$$
\L\Big\{R_{\frac{1}{2},\frac{1}{2}}(a,t);s\Big\} = \L\Big\{\D{}{0}{t}{-\frac{1}{2}} R_{0,\frac{1}{2}}(a,t);s\Big\} = \frac{\e^{-a \sqrt{s}}}{\sqrt{s}} = \L\Big\{\frac{\e^{-\frac{a^2}{4 t}}}{\sqrt{\pi t}};s\Big\},
$$
so that
\begin{equation}
\label{rodrigo-half-frac-int}
R_{\frac{1}{2},\frac{1}{2}}(a,t) = \D{}{0}{t}{-\frac{1}{2}} R_{0,\frac{1}{2}}(a,t) = \frac{\e^{-\frac{a^2}{4 t}}}{\sqrt{\pi t}}.
\end{equation}
\end{rem}

The next proposition will be useful when solving IBVPs for the time-fractional diffusion-wave equation.
\begin{prop}
Let $\mu \ge 0$, $0 < \nu \le 1$ and $a > 0$. Then
\begin{equation}
\label{rodrigo-imp-int}
R_{\mu + \nu,\nu}(a,t) = \int_a^\infty R_{\mu,\nu}(a',t) \, \d a', \quad t > 0.
\end{equation}
\end{prop}
\begin{proof}
We see from \eqref{rodrigo-fun} that
\begin{align*}
\L\Big\{\int_a^\infty R_{\mu,\nu}(a',t) \, \d a';s\Big\} & = \int_a^\infty \L\{R_{\mu.\nu}(a',t);s\} \, \d a' = \int_a^\infty s^{-\mu} \e^{-a' s^\nu} \, \d a' \\
& = s^{-(\mu + \nu)} \e^{-a s^\nu} = \L\{R_{\mu + \nu,\nu}(a,t);s\}
\end{align*}
and the result follows. 
\end{proof}

\begin{ex}
Assume that $\mu = \nu = \frac{1}{2}$. Then \eqref{rodrigo-imp-int} and \eqref{rodrigo-half-frac-int} imply that
\begin{equation}
\label{rodrigo-erf}
R_{1,\frac{1}{2}}(a,t) = \int_a^\infty \frac{\e^{-\frac{(a')^2}{4 t}}}{\sqrt{\pi t}} \, \d a' = \frac{2}{\sqrt{\pi}} \int_{\frac{a}{2 \sqrt{t}}}^\infty \e^{-z^2} \, \d z = \erfc\Big(\frac{a}{2 \sqrt{t}}\Big),
\end{equation}
where $\erfc$ is the complementary error function.
\end{ex}

The next lemma will be instrumental in solving various IVPs for the time-fractional diffusion-wave equation in Laplace space. 
\begin{lem}
\label{lem-uhat-eq-sol}
Suppose that $0 < \nu \le 1$. Let $\hat u(x,s)$ satisfy the inhomogeneous equation
\begin{equation}
\label{lem-uhat-eq}
\frac{\partial^2 \hat u}{\partial x^2}(x,s) - \frac{s^{2 \nu}}{\kappa} \hat u(x,s) = \hat G(x,s), \quad x \in \R
\end{equation}
and has bounded limits as $x \rightarrow \pm \infty$ for each~$s$, where $\hat G(x,s) = \L\{G(x,t);s\}$. Then 
\begin{equation*}
\hat u(x,s) = -\int_{-\infty}^\infty \frac{\sqrt{\kappa}}{2} s^{-\nu}\e^{-\frac{s^\nu}{\sqrt{\kappa}} \vert x - \xi \vert} \hat G(\xi,s) \, \d \xi.
\end{equation*}
\end{lem}
\begin{proof}
Using the variation of constants formula, the general solution of \eqref{lem-uhat-eq} can be written as
\begin{align*}
\hat u(x,s) & = \e^{\frac{s^\nu}{\sqrt{\kappa}} x} \Big[c_1 + \int_0^x \frac{\sqrt{\kappa}}{2} s^{-\nu} \e^{-\frac{s^\nu}{\sqrt{\kappa}} \xi} \hat G(\xi,s) \, \d \xi\Big] + \e^{-\frac{s^\nu}{\sqrt{\kappa}} x} \Big[c_2 - \int_0^x \frac{\sqrt{\kappa}}{2} s^{-\nu} \e^{\frac{s^\nu}{\sqrt{\kappa}} \xi} \hat G(\xi,s) \, \d \xi\Big]
\end{align*}
for arbitrary constants~$c_1$ and $c_2$. As we require $\hat u(\pm \infty,s)$ to be finite for each~$s$, we must have 
$$
c_1 = -\int_0^\infty \frac{\sqrt{\kappa}}{2} s^{-\nu} \e^{-\frac{s^\nu}{\sqrt{\kappa}} \xi} \hat G(\xi,s) \, \d \xi, \quad c_2 = -\int_{-\infty}^0 \frac{\sqrt{\kappa}}{2} s^{-\nu} \e^{\frac{s^\nu}{\sqrt{\kappa}} \xi} \hat G(\xi,s) \, \d \xi.
$$
Substituting these into the general solution and simplifying, we therefore deduce that
\begin{align*}
\hat u(x,s) & = -\int_x^\infty \frac{\sqrt{\kappa}}{2} s^{-\nu} \e^{\frac{s^\nu}{\sqrt{\kappa}} (x - \xi)} \hat G(\xi,s) \, \d \xi -\int_{-\infty}^x \frac{\sqrt{\kappa}}{2}  s^{-\nu} \e^{-\frac{s^\nu}{\sqrt{\kappa}} (x - \xi)} \hat G(\xi,s) \, \d \xi \\
& = -\int_x^\infty \frac{\sqrt{\kappa}}{2}  s^{-\nu} \e^{-\frac{s^\nu}{\sqrt{\kappa}} \vert \xi - x \vert} \hat G(\xi,s) \, \d \xi -\int_{-\infty}^x \frac{\sqrt{\kappa}}{2}  s^{-\nu} \e^{-\frac{s^\nu}{\sqrt{\kappa}} \vert x - \xi \vert} \hat G(\xi,s) \, \d \xi \\
& = -\int_{-\infty}^\infty \frac{\sqrt{\kappa}}{2} s^{-\nu} \e^{-\frac{s^\nu}{\sqrt{\kappa}} \vert x - \xi \vert} \hat G(\xi,s) \, \d \xi.
\end{align*}
\end{proof}

\section{Caputo time-fractional diffusion-wave equation on the real line}

Here we consider the inhomogeneous Caputo time-fractional diffusion-wave equation
\begin{equation}
\label{caputo-inhomogeneous}
{}_{0}^{C}D_{t}^{2 \nu} u = \kappa \frac{\partial^2 u}{\partial x^2} + F(x,t), \quad x \in \R, \quad t > 0
\end{equation}
with appropriate initial conditions as defined in \eqref{Phi-def} and \eqref{Psi-def}. We subdivide the analysis into two parts: $0 < \nu \le \frac{1}{2}$ and $\frac{1}{2} < \nu \le 1$.

\subsection{$0 < \nu \le \frac{1}{2}$}

In this subsection we assume the initial condition
\begin{equation}
\label{caputo-inhomogeneous-IC-case1}
u(x,0+) = f(x), \quad x \in \R.
\end{equation}
Since $0 < 2 \nu \le 1$, then $\ceil{2 \nu} = 1$. Taking the Laplace transform of \eqref{caputo-inhomogeneous} and using \eqref{lap-prop2}, we see that $\hat u(x,s)$ satisfies the inhomogeneous equation
\begin{equation*}
\frac{\partial^2 \hat u}{\partial x^2}(x,s) - \frac{s^{2 \nu}}{\kappa} \hat u(x,s) = -\frac{s^{2 \nu - 1}}{\kappa} f(x) - \frac{1}{\kappa} \hat F(x,s),
\end{equation*}
where $\hat F(x,s) = \L\{F(x,t);s\}$. Lemma~\ref{lem-uhat-eq-sol} gives
\begin{equation*}
\begin{split}
\hat u(x,s) & = -\int_{-\infty}^\infty \frac{\sqrt{\kappa}}{2} s^{-\nu} \e^{-\frac{s^\nu}{\sqrt{\kappa}} \vert x - \xi \vert} \Big[-\frac{s^{2 \nu - 1}}{\kappa} f(\xi) - \frac{1}{\kappa} \hat F(\xi,s)\Big] \, \d \xi,
\end{split}
\end{equation*}
whose inverse Laplace transform is
\begin{equation}
\label{caputo-case1-sol-temp}
\begin{split}
u(x,t) = \int_{-\infty}^\infty \frac{1}{2 \sqrt{\kappa}} \L^{-1}\{s^{-(1 - \nu)} \e^{-\frac{\vert x - \xi \vert}{\sqrt{\kappa}} s^\nu};t\} f(\xi) \, \d \xi + \int_{-\infty}^\infty \frac{1}{2 \sqrt{\kappa}} \L^{-1}\{s^{-\nu} \e^{-\frac{\vert x - \xi \vert}{\sqrt{\kappa}} s^\nu} \hat F(\xi,s);t\} \, \d \xi.
\end{split}
\end{equation}
We deduce from \eqref{rodrigo-fun} and \eqref{rodrigo-fun-spec} that
\begin{align*}
\L^{-1}\{s^{-(1 - \nu)} \e^{-\frac{\vert x - \xi \vert}{\sqrt{\kappa}} s^\nu};t\} & = R_{1 - \nu,\nu} \Big(\frac{\vert x - \xi \vert}{\sqrt{\kappa}},t\Big) = \D{}{0}{t}{-(1 - \nu)} R_{0,\nu}\Big(\frac{\vert x - \xi \vert}{\sqrt{\kappa}},t\Big), \\
\L^{-1}\{s^{-\nu} \e^{-\frac{\vert x - \xi \vert}{\sqrt{\kappa}} s^\nu};t\} & =  R_{\nu,\nu} \Big(\frac{\vert x - \xi \vert}{\sqrt{\kappa}},t\Big) = \D{}{0}{t}{-\nu} R_{0,\nu}\Big(\frac{\vert x - \xi \vert}{\sqrt{\kappa}},t\Big).
\end{align*}
Substituting these expressions into \eqref{caputo-case1-sol-temp} and invoking the convolution theorem, we have that the solution to the IVP for the inhomogeneous Caputo time-fractional diffusion equation ($0 < \nu \le \frac{1}{2}$) is given by
\begin{equation}
\label{caputo-case1-sol}
\begin{split}
u(x,t) & = \int_{-\infty}^\infty \frac{1}{2 \sqrt{\kappa}} R_{1 - \nu,\nu} \Big(\frac{\vert x - \xi \vert}{\sqrt{\kappa}},t\Big) f(\xi) \, \d \xi + \int_{-\infty}^\infty \int_0^t \frac{1}{2 \sqrt{\kappa}} R_{\nu,\nu} \Big(\frac{\vert x - \xi \vert}{\sqrt{\kappa}},t - \tau\Big) F(\xi,\tau) \, \d \tau \, \d \xi.
\end{split}
\end{equation}

\begin{rem}
Suppose that $F(x,t) = 0$ and $\nu = \frac{1}{2}$ in \eqref{caputo-inhomogeneous}, \eqref{caputo-inhomogeneous-IC-case1}. Then \eqref{caputo-case1-sol} simplifies to
$$
u(x,t) = \int_{-\infty}^\infty \frac{1}{2 \sqrt{\kappa}} R_{\frac{1}{2},\frac{1}{2}}\Big(\frac{\vert x - \xi \vert}{\sqrt{\kappa}},t\Big) f(\xi) \, \d \xi = \int_{-\infty}^\infty \frac{\e^{-\frac{(x - \xi)^2}{4 \kappa t}}}{\sqrt{4 \pi \kappa t}} f(\xi) \, \d \xi,
$$
where in the last step we used \eqref{rodrigo-half-frac-int} with $a = \frac{\vert x - \xi \vert}{\sqrt{\kappa}}$. This recovers the well-known Green's function solution to the IVP for the classical diffusion equation~\citep{St2008}.
\end{rem}

\begin{rem}
Let $F(x,t) = 0$ and $0 < \nu \le \frac{1}{2}$ in \eqref{caputo-inhomogeneous}, \eqref{caputo-inhomogeneous-IC-case1}. \citet{Ma1996} showed that the solution to this IVP is
$$
u(x,t) = \int_{-\infty}^\infty G_c(\xi,t) f(x - \xi) \, \d \xi = \int_{-\infty}^\infty G_c(x - \xi,t) f(\xi) \, \d \xi,
$$
where $G_c(x,t)$ is the fundamental solution of the Cauchy problem, which in turn can be written as
$$
G_c(x,t) = \frac{1}{2 \sqrt{\kappa} t^\nu} M\Big(\frac{\vert x \vert}{\sqrt{\kappa} t^\nu};\nu\Big)
$$
and $M(z;\nu)$ is the Mainardi function with the series representation~\citep{Ma1996}
$$
M(z;\nu) = \sum_{j = 0}^\infty \frac{(-1)^j z^j}{j! \Gamma(-\nu j + (1 - \nu))}, \quad 0 < \nu < 1.
$$
A related special function is the Wright function~\citep{Er1954}
$$
W(z;\alpha,\beta) = \sum_{j = 0}^\infty \frac{z^j}{j! \Gamma(\alpha j + \beta)}, \quad \alpha > -1, \quad \beta > 0.
$$
It follows that $M(z;\nu) = W(-z;-\nu,1 - \nu)$. On the other hand, we see from \eqref{caputo-case1-sol} that
$$
u(x,t) = \int_{-\infty}^\infty \frac{1}{2 \sqrt{\kappa}} R_{1 - \nu,\nu}\Big(\frac{\vert x - \xi \vert}{\sqrt{\kappa}},t\Big) f(\xi) \, \d \xi.
$$
We conclude that
$$
G_c(x - \xi,t) = \frac{1}{2 \sqrt{\kappa} t^\nu} M\Big(\frac{\vert x - \xi \vert}{\sqrt{\kappa} t^\nu};\nu\Big) = \frac{1}{2 \sqrt{\kappa}} R_{1 - \nu,\nu}\Big(\frac{\vert x - \xi \vert}{\sqrt{\kappa}},t\Big).
$$
The substitution~$a = \frac{\vert x - \xi \vert}{\sqrt{\kappa}}$ therefore leads to the interesting relation
$$
M\Big(\frac{a}{t^\nu};\nu\Big) = W\Big(-\frac{a}{t^\nu};-\nu,1 - \nu\Big) = t^\nu R_{1 - \nu,\nu}(a,t) = t^\nu \D{}{0}{t}{-(1 - \nu)} R_{0,\nu}(a,t).
$$
\end{rem}

\subsection{$\frac{1}{2} < \nu \le 1$}

Next we study \eqref{caputo-inhomogeneous} subject to the initial conditions
\begin{equation}
\label{caputo-inhomogeneous-IC-case2}
u(x,0+) = f(x), \quad D^1 u(x,0+) = g(x), \quad x \in \R.
\end{equation}
We see that $1 < 2 \nu \le 2$ and therefore $\ceil{2 \nu} = 2$. Taking the Laplace transform of \eqref{caputo-inhomogeneous} and using \eqref{lap-prop2} shows that $\hat u(x,s)$ satisfies the inhomogeneous equation
\begin{equation*}
\frac{\partial^2 \hat u}{\partial x^2}(x,s) - \frac{s^{2 \nu}}{\kappa} \hat u(x,s) = -\frac{s^{2 \nu - 1}}{\kappa} f(x) -\frac{s^{2 \nu - 2}}{\kappa} g(x) - \frac{1}{\kappa} \hat F(x,s),
\end{equation*}
where $\hat F(x,s) = \L\{F(x,t);s\}$. From Lemma~\ref{lem-uhat-eq-sol} we get
\begin{align*}
\hat u(x,s) & = -\int_{-\infty}^\infty \frac{\sqrt{\kappa}}{2}  s^{-\nu} \e^{-\frac{s^\nu}{\sqrt{\kappa}} \vert x - \xi \vert} \Big[-\frac{s^{2 \nu - 1}}{\kappa} f(\xi) -\frac{s^{2 \nu - 2}}{\kappa} g(\xi) - \frac{1}{\kappa} \hat F(\xi,s)\Big] \, \d \xi
\end{align*}
and its inverse Laplace transform is
\begin{equation}
\label{caputo-case2-sol-temp}
\begin{split}
u(x,t) & = \int_{-\infty}^\infty \frac{1}{2 \sqrt{\kappa}} \L^{-1}\{s^{-(1 - \nu)}\e^{-\frac{\vert x - \xi \vert}{\sqrt{\kappa}} s^\nu};t\} f(\xi) \, \d \xi + \int_{-\infty}^\infty \frac{1}{2 \sqrt{\kappa}} \L^{-1}\{s^{-(2 - \nu)}\e^{-\frac{\vert x - \xi \vert}{\sqrt{\kappa}} s^\nu};t\} g(\xi) \, \d \xi \\
& \quad {} + \int_{-\infty}^\infty \frac{1}{2 \sqrt{\kappa}} \L^{-1} \{s^{-\nu} \e^{-\frac{\vert x - \xi \vert}{\sqrt{\kappa}} s^\nu} \hat F(\xi,\tau);t\} \, \d \xi.
\end{split}
\end{equation}
Recalling \eqref{rodrigo-fun} and \eqref{rodrigo-fun-spec},
\begin{align*}
\L^{-1}\{s^{-(1 - \nu)}\e^{-\frac{\vert x - \xi \vert}{\sqrt{\kappa}} s^\nu};t\} & = R_{1 - \nu,\nu}\Big(\frac{\vert x - \xi \vert}{\sqrt{\kappa}},t\Big) = \D{}{0}{t}{-(1 - \nu)} R_\nu\Big(\frac{\vert x - \xi \vert}{\sqrt{\kappa}},t\Big), \\
\L^{-1}\{s^{-(2 - \nu)}\e^{-\frac{\vert x - \xi \vert}{\sqrt{\kappa}} s^\nu};t\} & =  R_{2 - \nu,\nu}\Big(\frac{\vert x - \xi \vert}{\sqrt{\kappa}},t\Big) = \D{}{0}{t}{-(2 - \nu)} R_\nu\Big(\frac{\vert x - \xi \vert}{\sqrt{\kappa}},t\Big), \\
\L^{-1} \{s^{-\nu} \e^{-\frac{\vert x - \xi \vert}{\sqrt{\kappa}} s^\nu};t\} & = R_{\nu,\nu}\Big(\frac{\vert x - \xi \vert}{\sqrt{\kappa}},t\Big) = \D{}{0}{t}{-\nu} R_\nu\Big(\frac{\vert x - \xi \vert}{\sqrt{\kappa}},t\Big).
\end{align*}
Substituting these expressions into \eqref{caputo-case2-sol-temp} and applying the convolution theorem, we derive the solution to the IVP for the inhomogeneous Caputo time-fractional wave equation ($\frac{1}{2} < \nu \le 1$) as
\begin{equation}
\label{caputo-case2-sol}
\begin{split}
u(x,t) & = \int_{-\infty}^\infty \frac{1}{2 \sqrt{\kappa}} R_{1 - \nu,\nu}\Big(\frac{\vert x - \xi \vert}{\sqrt{\kappa}},t\Big) f(\xi) \, \d \xi + \int_{-\infty}^\infty \frac{1}{2 \sqrt{\kappa}} R_{2 - \nu,\nu}\Big(\frac{\vert x - \xi \vert}{\sqrt{\kappa}},t\Big) g(\xi) \, \d \xi \\
& \quad {} + \int_{-\infty}^\infty \int_0^t \frac{1}{2 \sqrt{\kappa}} R_{\nu,\nu}\Big(\frac{\vert x - \xi \vert}{\sqrt{\kappa}},t - \tau\Big) F(\xi,\tau) \, \d \tau \, \d \xi.
\end{split}
\end{equation}

\begin{rem}
\label{cont-dep}
If $F(x,t) = 0$ and $g(x) = 0$, we observe that \eqref{caputo-case2-sol} becomes identical to \eqref{caputo-case1-sol}, thus showing the continuous dependence of the solution with respect to the parameter~$\nu$ as $\nu \rightarrow \frac{1}{2}^\pm$ (see also Remark~\ref{mainardi-ass}).
\end{rem}

\begin{ex}
\label{dalembert}
Suppose that $F(x,t) = 0$ and $\nu = 1$ in \eqref{caputo-inhomogeneous}, \eqref{caputo-inhomogeneous-IC-case2}, i.e. we are considering the classical wave equation. Recalling that $R_{0,1}(t,a) = \delta(t - a)$, then
\begin{align*}
\int_{-\infty}^\infty \frac{1}{2 \sqrt{\kappa}} R_{2 - 1,1}\Big(\frac{\vert x - \xi \vert}{\sqrt{\kappa}},t\Big) g(\xi) \, \d \xi & =\int_{-\infty}^\infty \frac{1}{2 \sqrt{\kappa}} \D{}{0}{t}{-(2 - 1)} R_{0,1}\Big(\frac{\vert x - \xi \vert}{\sqrt{\kappa}},t\Big) g(\xi) \, \d \xi \\
& = \frac{1}{2 \sqrt{\kappa}} \int_{-\infty}^\infty \Big[\int_0^t \delta\Big(\tau - \frac{\vert x - \xi \vert}{\sqrt{\kappa}}\Big) \, \d \tau\Big] g(\xi) \, \d \xi.
\end{align*}
Note that if $a > 0$ and $H$ is the Heaviside step function, then
$$
\int_0^t \delta(\tau - a) \, \d \tau = \int_{-\infty}^\infty \delta(\tau - a) H(\tau) H(t - \tau) \, \d \tau = H(a) H(t - a) = H(t - a).
$$
This implies that
\begin{align*}
\int_{-\infty}^\infty \frac{1}{2 \sqrt{\kappa}} R_{1,1}\Big(\frac{\vert x - \xi \vert}{\sqrt{\kappa}},t\Big) g(\xi) \, \d \xi & = \frac{1}{2 \sqrt{\kappa}} \int_{-\infty}^\infty H\Big(t - \frac{\vert x - \xi \vert}{\sqrt{\kappa}}\Big) g(\xi) \, \d \xi = \frac{1}{2 \sqrt{\kappa}} \int_{x - \sqrt{\kappa} t}^{x + \sqrt{\kappa} t} g(\xi) \, \d \xi.
\end{align*}
Moreover,
$$
\int_{-\infty}^\infty \frac{1}{2 \sqrt{\kappa}} R_{1 - 1,1}\Big(\frac{\vert x - \xi \vert}{\sqrt{\kappa}},t\Big) f(\xi) \, \d \xi = \int_{-\infty}^\infty \frac{1}{2 \sqrt{\kappa}} \delta\Big(t - \frac{\vert x- \xi \vert}{\sqrt{\kappa}}\Big) f(\xi) \, \d \xi.
$$
Using the scaling property of the Dirac delta function (i.e. $\delta(a x) = \frac{\delta(x)}{\vert a \vert}$ when $a \ne 0$), we see that
$$
\frac{1}{2 \sqrt{\kappa}} \delta\Big(t - \frac{\vert x- \xi \vert}{\sqrt{\kappa}}\Big) = \frac{1}{2 \sqrt{\kappa}} \delta\Big(\frac{1}{\sqrt{\kappa}} (\sqrt{\kappa} t - \vert x - \xi \vert)\Big) = \frac{1}{2 \sqrt{\kappa}} \frac{\delta(\sqrt{\kappa} t - \vert x - \xi \vert)}{\frac{1}{\sqrt{\kappa}}}
$$
and therefore
\begin{align*}
& \int_{-\infty}^\infty \frac{1}{2 \sqrt{\kappa}} R_{0,1}\Big(\frac{\vert x - \xi \vert}{\sqrt{\kappa}},t\Big) f(\xi) \, \d \xi = \int_{-\infty}^\infty \frac{1}{2} \delta(\sqrt{\kappa} t - \vert x - \xi \vert) f(\xi) \, \d \xi \\
& \qquad = \frac{1}{2} \int_{-\infty}^x \delta(\sqrt{\kappa} t - x + \xi) f(\xi) \, \d \xi + \frac{1}{2} \int_x^\infty \delta(\sqrt{\kappa} t + x - \xi) f(\xi) \, \d \xi.
\end{align*}
We have
\begin{align*}
& \frac{1}{2} \int_{-\infty}^x \delta(\sqrt{\kappa} t - x + \xi) f(\xi) \, \d \xi = \frac{1}{2} \int_{-\infty}^\infty \delta(\xi - (x - \sqrt{\kappa} t)) H(x - \xi) f(\xi) \, \d \xi \\
& \qquad = \frac{1}{2} H(x - (x - \sqrt{\kappa} t)) f(x - \sqrt{\kappa} t) = \frac{1}{2} f(x - \sqrt{\kappa} t)
\end{align*}
and
\begin{align*}
& \frac{1}{2} \int_x^\infty \delta(\sqrt{\kappa} t + x - \xi) f(\xi) \, \d \xi = \frac{1}{2} \int_{-\infty}^\infty \delta(-(\xi - x - \sqrt{\kappa} t)) H(-x + \xi) f(\xi) \, \d \xi \\
& \qquad = \frac{1}{2} \int_{-\infty}^\infty \frac{\delta(\xi - x - \sqrt{\kappa} t)}{\vert -1 \vert} H(-x + \xi) f(\xi) \, \d \xi = \frac{1}{2} H(-x + (x + \sqrt{\kappa} t)) f(x + \sqrt{\kappa} t) \\
& \qquad = \frac{1}{2} f(x + \sqrt{\kappa} t).
\end{align*}
Thus
$$
\int_{-\infty}^\infty \frac{1}{2 \sqrt{\kappa}} R_{0,1}\Big(\frac{\vert x - \xi \vert}{\sqrt{\kappa}},t\Big) f(\xi) \, \d \xi = \frac{1}{2} [f(x - \sqrt{\kappa} t + f(x + \sqrt{\kappa} t)].
$$
Substituting the above results into \eqref{caputo-case2-sol} recovers the well-known d'Alembert solution~\citep{St2008}
$$
u(x,t) = \frac{1}{2} [f(x - \sqrt{\kappa} t) + f(x + \sqrt{\kappa} t)] + \frac{1}{2 \sqrt{\kappa}} \int_{x - \sqrt{\kappa} t}^{x + \sqrt{\kappa} t} g(\xi) \, \d \xi.
$$
\end{ex}

\section{Riemann-Liouville diffusion-wave equation on the real line}

In this section we study the inhomogeneous Riemann-Liouville time-fractional diffusion-wave equation 
\begin{equation}
\label{riemann-liouville-inhomogeneous}
\D{}{0}{t}{2 \nu} u = \kappa \frac{\partial^2 u}{\partial x^2} + F(x,t), \quad x \in \R, \quad t > 0
\end{equation}
with appropriate initial conditions as defined in \eqref{Phi-def} and \eqref{Psi-def}.  We also subdivide the analysis into two parts: $0 < \nu \le \frac{1}{2}$ and $\frac{1}{2} < \nu \le 1$.

\subsection{$0 < \nu \le \frac{1}{2}$}

We assume the initial condition 
\begin{equation}
\label{riemann-liouville-inhomogeneous-IC-case1}
\D{}{0}{t}{-(1 - 2 \nu)} u(x,0+) = f(x), \quad x \in \R.
\end{equation}
Since $0 < 2 \nu \le 1$, we get $\ceil{2 \nu} = 1$. Taking the Laplace transform of \eqref{riemann-liouville-inhomogeneous} and using \eqref{lap-prop3}, we get that $\hat u(x,s)$ satisfies the inhomogeneous equation
\begin{equation*}
\frac{\partial^2 \hat u}{\partial x^2}(x,s) - \frac{s^{2 \nu}}{\kappa} \hat u(x,s) = -\frac{1}{\kappa} f(x) - \frac{1}{\kappa} \hat F(x,s),
\end{equation*}
where $\hat F(x,s) = \L\{F(x,t);s\}$. Lemma~\ref{lem-uhat-eq-sol} yields
\begin{equation*}
\begin{split}
\hat u(x,s) & = -\int_{-\infty}^\infty \frac{\sqrt{\kappa}}{2}  s^{-\nu} \e^{-\frac{s^\nu}{\sqrt{\kappa}} \vert x - \xi \vert} \Big[-\frac{1}{\kappa} f(\xi) - \frac{1}{\kappa} \hat F(\xi,s)\Big] \, \d \xi.
\end{split}
\end{equation*}
The inverse Laplace transform is
\begin{equation}
\label{riemann-liouville-case1-sol-temp}
u(x,t) = \int_{-\infty}^\infty \frac{1}{2 \sqrt{\kappa}} \L^{-1}\{s^{-\nu} \e^{-\frac{\vert x - \xi \vert}{\sqrt{\kappa}} s^\nu};t\} f(\xi) \, \d \xi + \int_{-\infty}^\infty \frac{1}{2 \sqrt{\kappa}} \L^{-1}\{s^{-\nu} \e^{-\frac{\vert x - \xi \vert}{\sqrt{\kappa}} s^\nu}\hat F(\xi,s);t\} \, \d \xi.
\end{equation}
Observing from \eqref{rodrigo-fun} and \eqref{rodrigo-fun-spec} that
$$
\L^{-1}\{s^{-\nu} \e^{-\frac{\vert x - \xi \vert}{\sqrt{\kappa}} s^\nu};t\} = R_{\nu,\nu}\Big(\frac{\vert x - \xi \vert}{\sqrt{\kappa}},t\Big) = \D{}{0}{t}{-\nu} R_\nu\Big(\frac{\vert x - \xi \vert}{\sqrt{\kappa}},t\Big),
$$
substituting this result into \eqref{riemann-liouville-case1-sol-temp} and invoking the convolution theorem, we finally obtain the solution to the IVP for the Riemann-Liouville time-fractional diffusion equation ($0 < \nu \le \frac{1}{2}$) as
\begin{equation}
\label{riemann-liouville-case1-sol}
\begin{split}
u(x,t) & = \int_{-\infty}^\infty \frac{1}{2 \sqrt{\kappa}} R_{\nu,\nu}\Big(\frac{\vert x - \xi \vert}{\sqrt{\kappa}},t\Big) f(\xi) \, \d \xi + \int_{-\infty}^\infty \int_0^t \frac{1}{2 \sqrt{\kappa}} R_{\nu,\nu}\Big(\frac{\vert x - \xi \vert}{\sqrt{\kappa}},t - \tau\Big) F(\xi,\tau) \, \d \tau \, \d \xi.
\end{split}
\end{equation}

\begin{rem}
If $F(x,t) = 0$ and $\nu = \frac{1}{2}$ in \eqref{riemann-liouville-inhomogeneous}, \eqref{riemann-liouville-inhomogeneous-IC-case1}, then \eqref{riemann-liouville-case1-sol} reduces to
$$
u(x,t) = \int_{-\infty}^\infty \frac{1}{2 \sqrt{\kappa}} R_{\frac{1}{2},\frac{1}{2}}\Big(\frac{\vert x - \xi \vert}{\sqrt{\kappa}},t\Big) f(\xi) \, \d \xi = \int_{-\infty}^\infty \frac{\e^{-\frac{(x - \xi)^2}{4 \kappa t}}}{\sqrt{4 \pi \kappa t}} f(\xi) \, \d \xi,
$$
using \eqref{rodrigo-half-frac-int} with $a = \frac{\vert x - \xi \vert}{\sqrt{\kappa}}$. As expected, this again recovers the well-known Green's function solution to the IVP for the classical diffusion equation~\citep{St2008}.
\end{rem}
 
\subsection{$\frac{1}{2} < \nu \le 1$}

Here we analyse \eqref{riemann-liouville-inhomogeneous} subject to the initial conditions
\begin{equation}
\label{riemann-liouville-inhomogeneous-IC-case2}
\D{}{0}{t}{-(2 - 2 \nu)} u(x,0+) = f(x), \quad D^1 \D{}{0}{t}{-(2 - 2 \nu)} u(x,0+) = g(x), \quad x \in \R.
\end{equation}
Taking the Laplace transform of \eqref{riemann-liouville-inhomogeneous} and using \eqref{lap-prop3}, we see that $\hat u(x,s)$ satisfies the inhomogeneous equation
\begin{equation*}
\frac{\partial^2 \hat u}{\partial x^2}(x,s) - \frac{s^{2 \nu}}{\kappa} \hat u(x,s) = -\frac{s}{\kappa} f(x) -\frac{1}{\kappa} g(x) - \frac{1}{\kappa} \hat F(x,s),
\end{equation*}
where $\hat F(x,s) = \L\{F(x,t);s\}$. From Lemma~\ref{lem-uhat-eq-sol} we have that
\begin{equation*}
\begin{split}
\hat u(x,s) & = -\int_{-\infty}^\infty \frac{\sqrt{\kappa}}{2} s^{-\nu} \e^{-\frac{s^\nu}{\sqrt{\kappa}} \vert x - \xi \vert} \Big[-\frac{s}{\kappa} f(\xi) -\frac{1}{\kappa} g(\xi) - \frac{1}{\kappa} \hat F(\xi,s)\Big] \, \d \xi
\end{split}
\end{equation*}
and whose inverse Laplace transform is
\begin{equation}
\label{riemann-liouville-case2-sol-temp}
\begin{split}
u(x,t) & = \int_{-\infty}^\infty \frac{1}{2 \sqrt{\kappa}} \L^{-1}\{s^{1 -\nu} \e^{-\frac{\vert x - \xi \vert}{\sqrt{\kappa}} s^\nu};t\} f(\xi) \, \d \xi + \int_{-\infty}^\infty \frac{1}{2 \sqrt{\kappa}} \L^{-1}\{s^{-\nu}\e^{-\frac{\vert x - \xi \vert}{\sqrt{\kappa}} s^\nu};t\} g(\xi) \, \d \xi \\
& \quad {} + \int_{-\infty}^\infty \frac{1}{2 \sqrt{\kappa}} \L^{-1}\{s^{-\nu} \e^{-\frac{\vert x - \xi \vert }{\sqrt{\kappa}} s^\nu}\hat F(\xi,s);t\} \, \d \xi.
\end{split}
\end{equation}
Recall that $\L\{y'(t);s\} = s \hat y(s) - y(0+)$; hence
$$
s \hat y(s) = \L\{y'(t) + y(0+) \delta(t);s\} \quad \text{or} \quad \L^{-1}\{s \hat y(s);t\} = y'(t) + y(0+) \delta(t).
$$
Choosing $y(t) = R_{\nu,\nu}(\frac{\vert x - \xi \vert}{\sqrt{\kappa}},t)$ in \eqref{rodrigo-fun} yields
$$
\hat y(s) = s^{-\nu} \e^{-\frac{\vert x - \xi \vert }{\sqrt{\kappa}} s^\nu}, \quad \L^{-1}\{s^{-\nu} \e^{-\frac{\vert x - \xi \vert }{\sqrt{\kappa}} s^\nu};t\} = R_{\nu,\nu}\Big(\frac{\vert x - \xi \vert}{\sqrt{\kappa}},t\Big).
$$
Also, \eqref{val-zero} implies that $y(0+) = 0$. Hence
$$
\L^{-1}\{s^{1 - \nu} \e^{-\frac{\vert x - \xi \vert }{\sqrt{\kappa}} s^\nu};t\} = \L^{-1}\{s \hat y(s);t\} = D^1 R_{\nu,\nu}\Big(\frac{\vert x - \xi \vert}{\sqrt{\kappa}},t\Big).
$$
Substituting the above results into \eqref{riemann-liouville-case2-sol-temp} and applying the convolution theorem, we deduce the solution to the IVP for the Riemann-Liouville time-fractional wave equation ($\frac{1}{2} < \nu \le 1$) to be
\begin{equation}
\label{riemann-liouville-case2-sol}
\begin{split}
u(x,t) & = \int_{-\infty}^\infty  \frac{1}{2 \sqrt{\kappa}}  D^1 R_{\nu,\nu}\Big(\frac{\vert x - \xi \vert}{\sqrt{\kappa}},t\Big) f(\xi) \, \d \xi + \int_{-\infty}^\infty  \frac{1}{2 \sqrt{\kappa}} R_{\nu,\nu}\Big(\frac{\vert x - \xi \vert}{\sqrt{\kappa}},t\Big) g(\xi) \, \d \xi \\
& \quad {} + \int_{-\infty}^\infty \int_0^t \frac{1}{2 \sqrt{\kappa}} R_{\nu,\nu}\Big(\frac{\vert x - \xi \vert}{\sqrt{\kappa}},t - \tau\Big) F(\xi,\tau) \, \d \tau \, \d \xi.
\end{split}
\end{equation}

\begin{rem}
If $F(x,t) = 0$ and $g(x) = 0$, we observe that \eqref{riemann-liouville-case2-sol} does not become identical to \eqref{riemann-liouville-case1-sol}; thus the continuous dependence of the solution with respect to the parameter~$\nu$ as $\nu \rightarrow \frac{1}{2}^\pm$ is not preserved for the Riemann-Liouville derivative (see Remarks~\ref{mainardi-ass} and \ref{cont-dep} for a comparison with the Caputo derivative).
\end{rem}

\begin{rem}
If $F(x,t) = 0$ and $\nu = 1$ in \eqref{riemann-liouville-inhomogeneous}, \eqref{riemann-liouville-inhomogeneous-IC-case2}, then following similar calculations as in Example~\ref{dalembert} we recover the d'Alembert solution for the classical wave equation~\citep{St2008} from \eqref{riemann-liouville-case2-sol}.
\end{rem}

\begin{rem}
\label{sum-fun-rem}
In \eqref{caputo-case1-sol}, \eqref{caputo-case2-sol}, \eqref{riemann-liouville-case1-sol} and \eqref{riemann-liouville-case2-sol} we make the useful observation that each can be expressed in the form (assuming the interchange of integrals is valid in the last term)
$$
u(x,t) = u_0(x,t) + \int_0^t \int_{-\infty}^\infty  \frac{1}{2 \sqrt{\kappa}} R_{\nu,\nu}\Big(\frac{\vert x - \xi \vert}{\sqrt{\kappa}},t - \tau\Big) F(\xi,\tau) \, \d \xi\, \d \tau,
$$
where $u_0(x,t)$ includes the initial data and the integral term includes the inhomogeneous term~$F(x,t)$. 
\end{rem}

\section{Initial-boundary value problems for the time-fractional diffusion-wave equation}

Let us now consider some IBVPs for the time-fractional diffusion-wave equation on the half-line, i.e.
\begin{equation}
\label{diff-wave-half-line}
D^{2 \nu} u = \kappa \frac{\partial^2 u}{\partial x^2}, \quad x > 0, \quad t > 0,
\end{equation}
where $D^{2 \nu}$ is either the Caputo operator~$\D{C}{0}{t}{2 \nu}$ or the Riemann-Liouville operator~$\D{}{0}{t}{2 \nu}$. The initial conditions are as before but changing $x \in \R$ to $x \ge 0$, and are summarised in Table~1.
\begin{table}[h]
\renewcommand{\arraystretch}{1.2}
\begin{center}
\begin{tabular}[h]{|l|c|c|c|}
\hline
Equation type & $D^{2 \nu}$ & $\nu$ & Initial condition \\
\hline 
Caputo diffusion & $\D{C}{0}{t}{2 \nu}$ & $0 < \nu \le \frac{1}{2}$ & \eqref{caputo-inhomogeneous-IC-case1} \\
\hline 
Caputo wave & $\D{C}{0}{t}{2 \nu}$ & $\frac{1}{2} < \nu \le 1$ & \eqref{caputo-inhomogeneous-IC-case2} \\
\hline 
Riemann-Liouville diffusion & $\D{}{0}{t}{2 \nu}$ & $0 < \nu \le \frac{1}{2}$ & \eqref{riemann-liouville-inhomogeneous-IC-case1} \\
\hline 
Riemann-Liouville wave & $\D{}{0}{t}{2 \nu}$ & $\frac{1}{2} < \nu \le 1$ & \eqref{riemann-liouville-inhomogeneous-IC-case2} \\
\hline 
\end{tabular}
\end{center}
\caption{Initial conditions for different IBVPs.}
\end{table}
We assume for simplicity the Dirichlet BC
\begin{equation}
\label{gen-BC}
u(0+,t) = h(t), \quad t > 0
\end{equation}
for some suitable function~$h(t)$. We remark that other types of linear BCs can also be considered~\citep{RoTh2021}.

Following the embedding method introduced by \citet{RoTh2021}, we embed the PDE~\eqref{diff-wave-half-line} and each of the initial conditions in Table~1 into an IVP for an inhomogeneous time-fractional diffusion-wave equation on the real line, i.e.
\begin{equation}
D^{2 \nu} \bar u = \kappa \frac{\partial^2 \bar u}{\partial x^2} + F(x,t), \quad x \in \R, \quad t > 0,
\end{equation}
where $F(x,t) = H(-x) \varphi(t)$ and $H$ is the Heaviside step function while $\varphi(t)$ is an arbitrary function to be determined so as to satisfy the BC~\eqref{gen-BC}. The initial conditions for $\bar u$ will involve two functions~$\bar f$ and $\bar g$ defined on $\R$ such that $\bar f(x) = f(x)$ and $\bar g(x) = g(x)$ when $x \ge 0$. Since $F(x,t) = 0$ when $x > 0$, we deduce that $\bar u$ will satisfy the PDE~\eqref{diff-wave-half-line} and the corresponding initial conditions in Table~1 when $x \ge 0$, for an arbitrary function~$\varphi(t)$. Then we set $\bar u(0+,t) = u(0+,t) = h(t)$ to determine $\varphi(t)$. We remark that the choice of the extensions~$\bar f$ and $\bar g$ is immaterial since the form of $\varphi(t)$ will get ``adjusted'' such that all conditions in the original IBVP will be satisfied in the end. 

Referring to Remark~\ref{sum-fun-rem}, the solution to all four IBVPs can be expressed, with $x > 0$, in the form
\begin{equation}
\label{IBVP-sol-form1}
u(x,t) = \bar u(x,t) = u_0(x,t) + \int_0^t \varphi(\tau) \int_{-\infty}^0 \frac{1}{2 \sqrt{\kappa}} R_{\nu,\nu}\Big(\frac{\vert x - \xi \vert}{\sqrt{\kappa}},t - \tau\Big) \, \d \xi \, \d \tau, 
\end{equation}
where $u_0(x,t)$ is given as follows:
\begin{enumerate}
\item[(i)] Caputo time-fractional diffusion equation
$$
u_0(x,t) = \int_{-\infty}^\infty \frac{1}{2 \sqrt{\kappa}} R_{1 - \nu,\nu}\Big(\frac{\vert x - \xi \vert}{\sqrt{\kappa}},t\Big) \bar f(\xi) \, \d \xi;
$$
\item[(ii)] Caputo time-fractional wave equation
$$
u_0(x,t) = \int_{-\infty}^\infty \frac{1}{2 \sqrt{\kappa}} R_{1 - \nu,\nu}\Big(\frac{\vert x - \xi \vert}{\sqrt{\kappa}},t\Big) \bar f(\xi) \, \d \xi + \int_{-\infty}^\infty \frac{1}{2 \sqrt{\kappa}} R_{2 - \nu,\nu}\Big(\frac{\vert x - \xi \vert}{\sqrt{\kappa}},t\Big) \bar g(\xi) \, \d \xi;
$$
\item[(iii)] Riemann-Liouville time-fractional diffusion equation
$$
u_0(x,t) = \int_{-\infty}^\infty \frac{1}{2 \sqrt{\kappa}} R_{\nu,\nu}\Big(\frac{\vert x - \xi \vert}{\sqrt{\kappa}},t\Big) \bar f(\xi) \, \d \xi;
$$
\item[(iv)] Riemann-Liouville time-fractional wave equation
$$
u_0(x,t) = \int_{-\infty}^\infty  \frac{1}{2 \sqrt{\kappa}}  D^1 R_{\nu,\nu}\Big(\frac{\vert x - \xi \vert}{\sqrt{\kappa}},t\Big) \bar f(\xi) \, \d \xi + \int_{-\infty}^\infty  \frac{1}{2 \sqrt{\kappa}} R_{\nu,\nu}\Big(\frac{\vert x - \xi \vert}{\sqrt{\kappa}},t\Big) \bar g(\xi) \, \d \xi.
$$
\end{enumerate}
Since $-\infty < \xi \le 0 < x$ in the integral term of \eqref{IBVP-sol-form1}, it is possible to rewrite
\begin{align*}
\int_{-\infty}^0 \frac{1}{2 \sqrt{\kappa}} R_{\nu,\nu}\Big(\frac{\vert x - \xi \vert}{\sqrt{\kappa}},t - \tau\Big) \, \d \xi & = \frac{1}{2} \int_{\frac{x}{\sqrt{\kappa}}}^\infty R_{\nu,\nu}(z,t- \tau) \, \d z = \frac{1}{2} R_{2 \nu,\nu}\Big(\frac{x}{\sqrt{\kappa}},t - \tau\Big)
\end{align*}
using \eqref{rodrigo-imp-int}. Hence \eqref{IBVP-sol-form1} is equivalent to
\begin{equation}
\label{IBVP-sol-form2}
u(x,t) = u_0(x,t) + \frac{1}{2} \int_0^t \varphi(\tau) R_{2 \nu,\nu}\Big(\frac{x}{\sqrt{\kappa}},t - \tau\Big) \, \d \tau.
\end{equation}

As mentioned previously, \eqref{IBVP-sol-form2} satisfies the PDE~\eqref{diff-wave-half-line} and the initial conditions, for an arbitrary~$\varphi(t)$. We now look for $\varphi(t)$ such that the BC~\eqref{gen-BC} is satisfied. Then
$$
h(t) = u(0+,t) = u_0(0+,t) + \frac{1}{2} \int_0^t \varphi(\tau) R_{2 \nu,\nu}(0+,t - \tau) \, \d \tau
$$
from \eqref{IBVP-sol-form2}. Taking the Laplace transform and using the convolution theorem and \eqref{rodrigo-fun},
$$
\hat h(s) = \hat u_0(0+,s) + \frac{1}{2} s^{-2 \nu} \hat \varphi(s) \quad \text{or} \quad s^{-2 \nu} \hat \varphi(s) = 2 [\hat h(s) - \hat u_0(0+,s)].
$$
Evaluating the inverse Laplace transform and recalling \eqref{lap-prop1}, we obtain
$$
\D{}{0}{t}{-2 \nu} \varphi(t) = 2 [h(t) - u_0(0+,t)].
$$
Therefore the solution of the IBVP is
\begin{equation}
\label{IBVP-sol-form3}
u(x,t) = u_0(x,t) + \frac{1}{2} \int_0^t \varphi(\tau) R_{2 \nu,\nu}\Big(\frac{x}{\sqrt{\kappa}},t - \tau\Big) \, \d \tau, \quad \D{}{0}{t}{-2 \nu} \varphi(t) = 2 [h(t) - u_0(0+,t)].
\end{equation}
Note the $u_0(x,t)$ is known from (i)-(iv) above and $h(t)$ is given.

\begin{ex}
Suppose that $\nu = \frac{1}{2}$, $f(x) = 0$ for all $x \ge 0$ and $h(t) = 1$ for all $t > 0$. These assumptions lead to an IVBP for the classical heat equation on the half-line. Define $\bar f(x) = 0$ for all $x \in \R$, so that $\bar f(x) = f(x)$ for $x \ge 0$; hence $u_0(x,t) = 0$ in (i). From the second relation in \eqref{IBVP-sol-form3} we see that $\varphi(t)$ is such that $\D{}{0}{t}{-1} \varphi(t) = 2$, and so its Laplace transform is $\frac{\hat \varphi(s)}{s} = \frac{2}{s}$ or $\varphi(t) = 2 \delta(t)$. Therefore the first equation in \eqref{IBVP-sol-form3} implies that
$$
u(x,t) = \int_0^t \delta(\tau) R_{1,\frac{1}{2}}\Big(\frac{x}{\sqrt{\kappa}},t - \tau\Big) \, \d \tau = \int_0^t \delta(\tau) \erfc\Big(\frac{x}{2 \sqrt{\kappa (t - \tau)}}\Big) \, \d \tau = \erfc\Big(\frac{x}{2 \sqrt{\kappa t}}\Big),
$$
where we used \eqref{rodrigo-erf} in the penultimate step. This is of course a well-known result~\citep{St2008}.
\end{ex}

\begin{ex}
Let us now return to the Cauchy and signalling problems studied by \citet{Ma1996}. It was shown there that the solution of the Cauchy problems~\eqref{mainardi-cauchy-diff} and \eqref{mainardi-cauchy-wave} is given by
$$
u(x,t) = \int_{-\infty}^\infty G_c(\xi,t) f(x - \xi) \, \d \xi,
$$
while the solution of the signalling problems~\eqref{mainardi-signal-diff} and \eqref{mainardi-signal-wave} is expressed as
$$
u(x,t) = \int_0^t G_s(x,\tau) h(t - \tau) \, \d \tau.
$$
The fundamental solutions~$G_c(x,t)$ and $G_s(x,t)$ in Laplace transform space are
\begin{equation}
\label{G-lap}
\hat G_c(x,s) = \frac{1}{2 \sqrt{\kappa}} s^{-(1 - \nu)} \e^{-\frac{\vert x \vert}{\sqrt{\kappa}} s^\nu}, \quad x \in \R, \quad \hat G_s(x,s) = \e^{-\frac{x}{\sqrt{\kappa}} s^\nu}, \quad x > 0.
\end{equation}
It directly follows that 
$$
\frac{\partial \hat G_s}{\partial s}(x,s) = -2 \nu x \hat G_c(x,s), \quad x > 0.
$$
Using standard properties of the Laplace transform, a reciprocity relation~\citep{Ma1996} can be deduced:
$$
x G_c(x,t) = \frac{t}{2 \nu} G_s(x,t), \quad x > 0.
$$
In fact, from \eqref{G-lap}, \eqref{rodrigo-fun} and \eqref{rodrigo-fun-spec} we have
\begin{align*}
G_c(x,t) & = \frac{1}{\sqrt{\kappa}} R_{1 - \nu,\nu}\Big(\frac{\vert x \vert}{\sqrt{\kappa}},t\Big) = \frac{1}{\sqrt{\kappa}} \D{}{0}{t}{-(1 - \nu)} R_{0,\nu}\Big(\frac{\vert x \vert}{\sqrt{\kappa}},t\Big), \quad x \in \R, \\
G_s(x,t) & = R_{0,\nu}\Big(\frac{x}{\sqrt{\kappa}},t\Big), \quad x > 0.
\end{align*}
Hence here we deduce a different type of relation between the fundamental solutions, namely
$$
G_c(x,t) = \frac{1}{\sqrt{\kappa}} \D{}{0}{t}{-(1 - \nu)} G_s(x,t), \quad x > 0.
$$
\end{ex}

\section{Discussion}

In this brief discussion we explore the possibility of generating probability distributions from a time-fractional diffusion equation ($0 < \nu \le \frac{1}{2}$). This is in the same spirit as in \citep{MaPaGo2007}.

Consider an It\^{o} process $\{X_t : t \ge 0\}$ which satisfies the stochastic differential equation
$$
\d X_t = \mu(X_t,t) \, \d t + \sigma(X_t,t) \, \d W_t,
$$
where $\mu(x,t)$ and $\sigma(x,t)$ are given deterministic functions and $\{W_t : t \ge 0\}$ is the driving Wiener process.
The Fokker-Planck equation for the probability density function~$p(x,t)$ of $X_t$ is given by the PDE
$$
D^1 p = -\frac{\partial}{\partial x}[\mu(x,t) p] + \frac{1}{2} \frac{\partial^2}{\partial x^2} [\sigma(x,t)^2 p], \quad x \in \R, \quad t > 0.
$$

As a special case, a Wiener process is also an It\^{o} process since we may take $\mu(x,t) = 0$ and $\sigma(x,t) = 1$, so that $X_t = W_t$. In this case the Fokker-Planck equation is the classical diffusion equation with $\kappa = \frac{1}{2}$:
$$
D^1 p = \frac{1}{2} \frac{\partial^2 p}{\partial x^2}, \quad x \in \R, \quad t > 0.
$$
If $p(x,0) = \delta(x)$ for $x \in \R$, then the probability density function~$p(x,t)$ of $X_t = W_t \sim N(0,t)$ is of course
$$
p(x,t) = \int_{-\infty}^\infty \frac{\e^{-\frac{(x - \xi)^2}{2 t}}}{\sqrt{2 \pi t}} \delta(\xi) \, \d \xi = \frac{\e^{-\frac{x^2}{2 t}}}{\sqrt{2 \pi t}}.
$$
It can be shown that
$$
\int_{-\infty}^\infty p(x,t) \, \d x = 1, \quad t > 0.
$$
Therefore we deduce that the IVP for the classical diffusion equation can be used to generate a normal probability distribution with mean~$0$ and variance~$t$ for the continuous random variable~$X_t$ for each $t > 0$. Here we wish to see if a time-fractional diffusion equation can also generate a more general probability distribution for some continuous random variable~$X_t$ for each $t > 0$.

With the above example as motivation, consider the following IVP for a Caputo time-fractional diffusion equation:
\begin{align*}
& \D{C}{0}{t}{2 \nu} p = \kappa \frac{\partial^2 p}{\partial x^2}, \quad x \in \R, \quad t  > 0, \\
& p(x,0) = \delta(x), \quad x \in \R.
\end{align*}
From \eqref{caputo-case1-sol}, \eqref{rodrigo-fun} and \eqref{rodrigo-fun-spec} we obtain
$$
p(x,t) = \frac{1}{2 \sqrt{\kappa}} R_{1 - \nu,\nu}\Big(\frac{\vert x \vert}{\sqrt{\kappa}},t\Big) = \frac{1}{2 \sqrt{\kappa}} \D{}{0}{t}{-(1 - \nu)} R_\nu\Big(\frac{\vert x \vert}{\sqrt{\kappa}},t\Big), \quad \hat p(x,s) = \frac{1} {\sqrt{\kappa}} s^{-(1 - \nu)} \e^{-\frac{\vert x \vert}{\sqrt{\kappa}} s^\nu},
$$
so that
\begin{align*}
\int_{-\infty}^\infty \hat p(x,s) \, \d x & = \frac{1}{\sqrt{\kappa}} s^{-(1 - \nu)} \int_0^\infty \e^{-\frac{s^\nu}{\sqrt{\kappa}} \vert x \vert} \, \d x = s^{-1}
\end{align*}
and
$$
\int_{-\infty}^\infty p(x,t) \, \d x = \int_{-\infty}^\infty \L^{-1}\{\hat p(x,s);t\} \, \d x = \L^{-1}\Big\{\int_{-\infty}^\infty \hat p(x,s) \, \d x;t\Big\} = \L^{-1}\{s^{-1};t\} = 1.
$$
Hence 
\begin{equation}
\label{pdf-caputo}
p(x,t) = \frac{1}{2 \sqrt{\kappa}} R_{1 - \nu,\nu}\Big(\frac{\vert x \vert}{\sqrt{\kappa}},t\Big) = \frac{1}{2 \sqrt{\kappa}} \D{}{0}{t}{-(1 - \nu)} R_\nu\Big(\frac{\vert x \vert}{\sqrt{\kappa}},t\Big)
\end{equation} 
is the probability density function of some continuous random variable~$X_t$ for each $t > 0$. In the special case that $\nu = \frac{1}{2}$ we have already seen that $X_t \sim N(0,t)$.

\begin{rem}
Recall that a generalised Gaussian distribution for a random variable~$X$ has a three-parameter probability density function
\begin{equation}
\label{ggd}
f_X(x) = \frac{b}{2 a \Gamma(\frac{1}{b})} \e^{-(\frac{x - c}{a})^b}, \quad a, b > 0, \quad c \in \R.
\end{equation}
This is a parametric family of symmetric distributions and includes the normal and Laplace distributions. It is also known that
$$
E(X) = c, \quad \Var(X) = \frac{a^2 \Gamma(\frac{3}{b})}{\Gamma(\frac{1}{b})}.
$$
For example, if $a = \sqrt{2 t}$, $b = 2$, $c = 0$ and $X = X_t$, then \eqref{ggd} becomes
$$
f_{X_t}(x) = \frac{\e^{-\frac{x^2}{2 t}}}{\sqrt{2 \pi t}}, \quad X_t \sim N(0,t).
$$
The probability density function~$p(x,t)$ in \eqref{pdf-caputo} is also symmetric and the normal distribution is a special case. However, in general it is not the same as a generalised Gaussian distribution. 
\end{rem}

Let us repeat the above calculations for the Riemann-Liouville time-fractional diffusion equation, i.e. consider the IVP
\begin{equation*}
\begin{split}
& \D{}{0}{t}{2 \nu} p = \kappa \frac{\partial^2 p}{\partial x^2}, \quad x \in \R, \quad t > 0, \\
& \D{}{0}{t}{-(1 - 2\nu)} p(x,0) = \delta(x), \quad x \in \R.
\end{split}
\end{equation*}
Then \eqref{riemann-liouville-case1-sol}, \eqref{rodrigo-fun} and \eqref{rodrigo-fun-spec} yield
$$
p(x,t) = \frac{1}{2 \sqrt{\kappa}} R_{\nu,\nu}\Big(\frac{\vert x \vert}{\sqrt{\kappa}},t\Big) = \frac{1}{2 \sqrt{\kappa}} \D{}{0}{t}{-\nu} R_\nu\Big(\frac{\vert x \vert}{\sqrt{\kappa}},t\Big), \quad \hat p(x,s) = \frac{1}{2 \sqrt{\kappa}} s^{-\nu} \e^{-\frac{s^\nu}{\sqrt{\kappa}} \vert x \vert}.
$$
Moreover, 
$$
\int_{-\infty}^\infty \hat p(x,s) \, \d x = \frac{1}{\sqrt{\kappa}} s^{-\nu} \int_0^\infty \e^{-\frac{s^\nu}{\sqrt{\kappa}} \vert x \vert} \, \d x = s^{-2 \nu}
$$
and
$$
\int_{-\infty}^\infty p(x,t) \, \d x = \int_{-\infty}^\infty \L^{-1}\{\hat p(x,s);t\} \, \d x = \L^{-1}\Big\{\int_{-\infty}^\infty \hat p(x,s) \, \d x;t\Big\} = \L^{-1}\{s^{-2 \nu};t\} = \frac{t^{2 \nu - 1}}{\Gamma(2 \nu)}.
$$
Hence $\int_{-\infty}^\infty p(x,t) \, \d x = 1$ if and only if $\nu = \frac{1}{2}$. Therefore the Riemann-Liouville time-fractional diffusion equation can be used to generate probability distributions only when $\nu = \frac{1}{2}$.

\section{Concluding remarks}

In this article we considered the time-fractional diffusion-wave equation for both Caputo and Riemann-Liouville fractional time derivatives.   Using the Laplace transform, we studied IVPs for this equation and found the formal solutions. For this purpose we defined a useful auxiliary function and studied some of its properties, including a derivation of fractional-order integral and ordinary differential equations that it satisfies. Then we formulated IVBPs for the time-fractional diffusion-wave equation, applied an embedding approach and used the results for IVPs to find the formal solutions of the IBVPs. Finally, we explored the possibility of using the time-fractional diffusion equation to generate probability distributions. We showed that the Caputo operator is able to generate probability density functions for any $0 < \nu \le \frac{1}{2}$ but for the Riemann-Liouville operator it is only possible when $\nu = \frac{1}{2}$. Current works in progress by the author are the extension of the results of this article to a space-fractional diffusion-wave equation, as in \citep{MaPaGo2007}, and the consideration of free boundary problems for the time-fractional diffusion-wave equation following the idea in \citet{RoTh2021}.

\bigskip
\noindent
{\Large \bf Appendix}

\begin{appendix}
\bigskip
Here we will prove \eqref{rodrigo-real-int}. Suppose that $0 < \nu \le \frac{1}{2}$ and define $\hat f(s) = \e^{-a s^\nu}$. By the complex inversion formula, the inverse Laplace transform is
$$
R_{0,\nu}(a,t) = f(t) = \L^{-1}\{\e^{-a s^\nu};t\} = \frac{1}{2 \pi \i} \int_{\gamma - \i \infty}^{\gamma + \i \infty} \e^{s t - a s^\nu} \, \d s.
$$
Since $s = 0$ is a branch point of the integrand, we write
\begin{align*}
\frac{1}{2 \pi \i} \oint_C \e^{s t - a s^\nu} \, \d s & = \frac{1}{2 \pi \i} \int_{\gamma - \i T}^{\gamma + \i T} \e^{s t - a s^\nu} \, \d s + \frac{1}{2 \pi \i} \int_{BD} \e^{s t - a s^\nu} \, \d s + \frac{1}{2 \pi \i} \int_{DE} \e^{s t - a s^\nu} \, \d s \\
& \quad {} + \frac{1}{2 \pi \i} \int_{EFG} \e^{s t - a s^\nu} \, \d s + \frac{1}{2 \pi \i} \int_{GH} \e^{s t - a s^\nu} \, \d s + \frac{1}{2 \pi \i} \int_{HA} \e^{s t - a s^\nu} \, \d s,
\end{align*}
where $C$ is the contour in Figure~3.
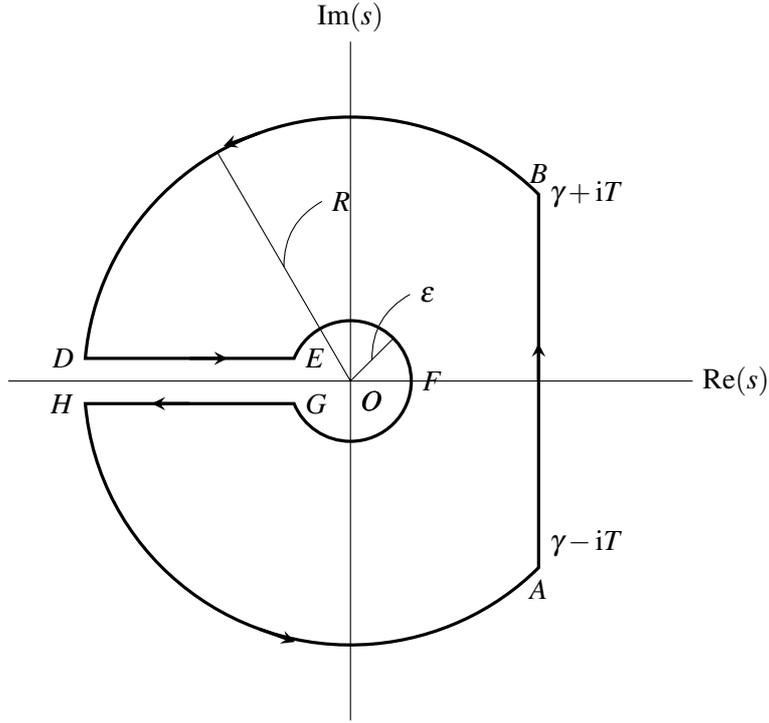
\begin{figure}[ht]
\centering
\begin{center}
\begin{tikzpicture}[pics/arrow/.style = {code = {\draw[very thick, -stealth] (0,0) -- (0.5,0);}},
   declare function = {r = 0.8; R = 3.5; h = 0.3; gamma(\x) = 180 - asin(h/\x);},
   s/.style = {sloped, allow upside down, pos = #1}, s/.default = 0.5]
   \draw[very thick] 
   ({gamma(r)}:r) node[right]{$E$} 
   arc[start angle = {gamma(r)}, end angle = {0}, radius = r] node[right]{$F$}
   arc[start angle = {0}, end angle = {-gamma(r)}, radius = r] node[right]{$G$} 
   -- pic[s]{arrow} ({-gamma(R)}:R) node[left]{$H$}   
   arc[start angle = {-gamma(R)}, end angle = -45, radius = R] pic[s]{arrow} node[below]{$A$} node[above right]{$\gamma - \mathrm{i} T$}
   -- pic[s]{arrow} (45:R) node[above]{$B$} node[right]{$\gamma + \mathrm{i} T$} 
   arc[start angle = 45, end angle = {gamma(R)}, radius = R] pic[s]{arrow} node[left]{$D$}
   -- pic[s]{arrow} cycle
   ;
   
   \draw (-R-1, 0) -- (R+1, 0) node[right]{$\mathrm{Re}(s)$}
   (0, -R-1) -- (0, R+1) node [above]{$\mathrm{Im}(s)$} 
   (0, 0) node[below right]{$O$} -- coordinate (aux) (45:r) (aux) to[bend left] ++ (60:1) node[right]{$\epsilon$}
   (0, 0) node[below right]{$O$} -- coordinate (aux) (120:R) (aux) to[bend left] ++ (60:1) node[right]{$R$}
   ;
\end{tikzpicture}
\end{center}
\caption{A modifed Bromwich contour~$C$ to avoid the branch point at $s = 0$.}
\end{figure}
Since the only singularity at $s = 0$ is outside $C$, the integral on the left is zero. 

Suppose that $s = R \e^{\i \theta}$, so that $s^\nu = R^\nu \e^{\i \theta \nu}$. Along $BD$ we have $0 < \theta < \pi$ or $0 < \theta \nu < \pi \nu \le \frac{\pi}{2}$. Similarly, along $HA$ we see that $-\pi < \theta < 0$ or $-\frac{\pi}{2} \le -\pi \nu < \theta \nu < 0$. In either case we have $\cos(\theta \nu) > 0$. Hence 
$$
\vert \hat f(s) \vert = \vert \e^{-a R^\nu s^{\i \theta \nu}} \vert = \vert \e^{-a R^\nu [\cos(\theta \nu) + \i \sin(\theta \nu)]} \vert = \e^{-a R^\nu \cos(\theta \nu)} < \frac{1}{R^\nu}
$$
for $R$ sufficiently large since $a > 0$ and $\cos(\theta \nu) > 0$. Therefore the integrals along $BD$ and $HA$ tend to zero as $R \rightarrow \infty$. Note that the above argument breaks down if $\frac{1}{2} < \nu \le 1$ since we cannot guarantee that $\cos(\theta \nu) > 0$. 

We have
\begin{align*}
f(t) & = \frac{1}{2 \pi \i} \int_{\gamma - \i \infty}^{\gamma + \i \infty} \e^{s t - a s^\nu} \, \d s = -\lim_{\stackrel{R \rightarrow \infty}{\epsilon \rightarrow 0}} \frac{1}{2 \pi \i} \Big[\int_{DE} \e^{s t - a s^\nu} \, \d s + \int_{EFG} \e^{s t - a s^\nu} \, \d s + \int_{GH} \e^{s t - a s^\nu} \, \d s\Big].
\end{align*}
Along $DE$ we have $s = x \e^{\i \pi} = -x$ as $x$ varies from $R$ to $\epsilon$. Then $s^\nu = x^\nu \e^{\i \pi \nu}$ and
$$
\int_{DE} \e^{s t - a s^\nu} \, \d s = -\int_R^{\epsilon} \e^{x \e^{\pi i} t - a x^\nu \e^{\i \pi \nu}} \d x = \int_\epsilon^R \e^{-x t - a x^\nu \e^{\i \pi \nu}} \, \d x.
$$
Along $GH$ we have $s = x \e^{-\i \pi} = -x$ as $x$ varies from $\epsilon$ to $R$. Then $s^\nu = x^\nu \e^{-\i \pi \nu}$ and
$$
\int_{GH} \e^{s t - a s^\nu} \, \d s = -\int_{\epsilon}^R \e^{x \e^{-\pi i} t - a x^\nu \e^{-\i \pi \nu}} \d x = -\int_\epsilon^R \e^{-x t - a x^\nu \e^{-\i \pi \nu}} \, \d x.
$$
Along $EFG$ we see that $s = \epsilon \e^{\i \theta}$ as $\theta$ varies from $\pi$ to $-\pi$, $s^\nu = \epsilon^\nu \e^{\i \theta \nu}$ and $\d s = \i \epsilon \e^{\i \theta} \, \d \theta$. Then 
$$
\int_{EFG} \e^{s t - a s^\nu} \, \d s = \int_\pi^{-\pi} \e^{\epsilon \e^{\i \theta} t - a \epsilon^\nu \e^{\i \theta \nu}} \i \epsilon \e^{\i \theta} \, \d \theta,
$$
which tends to zero as $\epsilon \rightarrow 0^+$. Thus, when $0 < \nu \le \frac{1}{2}$, we deduce that
\begin{align*}
R_\nu(a,t) = f(t) & = -\frac{1}{2 \pi \i} \Big[\int_0^\infty \e^{-x t - a x^\nu \e^{\i \pi \nu}} \, \d x - \int_0^\infty \e^{-x t - a x^\nu \e^{-\i \pi \nu}} \, \d x\Big] \\
& = -\frac{1}{2 \pi \i} \int_0^\infty \e^{-x t} \{\e^{-a x^\nu [\cos(\pi \nu) + \i \sin(\pi \nu)]} - \e^{-a x^\nu [\cos(\pi \nu) - \i \sin(\pi \nu)]}\} \, \d x \\
& =  -\frac{1}{2 \pi \i} \int_0^\infty \e^{-x t} \e^{-a x^\nu \cos(\pi \nu)} (-2 \i) \sin(a x^\nu \sin(\pi \nu)) \, \d x \\
& = \frac{1}{\pi} \int_0^\infty \e^{-t x} \e^{-a \cos(\pi \nu) x^\nu} \sin(a \sin (\pi \nu) x^\nu) \, \d x.
\end{align*} 
\end{appendix}

\end{document}